\documentclass[twoside,a4paper]{article}
\usepackage[english]{babel}
\usepackage[latin1]{inputenc}
\usepackage{amsmath,amssymb,amsthm}
\usepackage{amsfonts}
\usepackage{latexsym}
\usepackage{graphics, graphicx}
\usepackage{upref}
\usepackage{psfrag,rotating}
\usepackage{t1enc}
\usepackage[]{color}
\usepackage{subfig}
\usepackage[T1]{fontenc} 
\usepackage{lmodern}  
\usepackage{comment}     
\usepackage[numbers]{natbib}     % bibtex package
\usepackage{multirow}
\usepackage{caption}
\usepackage{float}
\usepackage[arrow, matrix, curve]{xy}
\usepackage{tikz}
\usepackage{pgfplots}
\usepackage{multicol}
\usepackage[section]{placeins}
\pgfplotsset{compat=1.3}
\usetikzlibrary{plotmarks}
\usepackage{enumerate}
\usepackage{todonotes}
\usepackage{hyperref}

\allowdisplaybreaks[1]

\usepackage[most]{tcolorbox}
\newcounter{common}[section]

\numberwithin{equation}{section}

\newcommand{\nz}{\mathbb{N}}
\newcommand{\rz}{\mathbb{R}}

\newcommand{\intO}{\int_\Omega}

\def\cringle{\mathaccent"7017 }

\newcommand{\operator}{{\cal F}}

\newcommand{\Lltwo}{\boldsymbol{L}^2}

\newcommand{\Vv}{\mathbf{V}}
\newcommand{\Vvh}{\mathbf{V}_h}

\newcommand{\Wpone}{\cringle{W}^1_p(\Omega)}
\newcommand{\Wponep}{\cringle{W}^1_{p\dual}(\Omega)}
\newcommand{\Wptwo}{W^2_p(\Omega)}

\newcommand{\keff}{k(\phi)}

\newcommand{\muu}{\mu(\phi)}

\newcommand{\phid}{\phi_D}

\newcommand{\uu}{\mathbf{u}}

\newcommand{\vv}{\mathbf{v}}
\newcommand{\ww}{\mathbf{w}}

\newcommand{\ee}{\mathbf{e}}

\newcommand{\jjp}{\mathbf{j}}
\newcommand{\jjR}{\mathbf{j}_R}

\newcommand{\jj}{\mathbf{j}}

\newcommand{\gb}{\mathbf{g}}

\newcommand{\real}{\ensuremath{\mathbb{R}}}

\newcommand{\bit}{\begin{itemize}}
\newcommand{\eit}{\end{itemize}}
\newcommand{\beq}{\begin{equation}}
\newcommand{\eeq}{\end{equation}}

\newcommand{\dual}{'}

\newcommand*\Div{\nabla\hskip -2pt \cdot}

\newcommand{\Rey}{\ensuremath\mathrm{R\hskip -1 pt e}}

\newcommand{\Sc}{\ensuremath\mathrm{S\hskip -1 pt c}}
\newcommand{\Pra}{\ensuremath\mathrm{P\hskip -1 pt r}}
\newcommand{\Le}{\ensuremath\mathrm{L\hskip -1 pt e}}

\newcommand{\NBT}{N_{BT}}
\newcommand{\Scf}{\ensuremath\mathrm{S\hskip -1 pt c \hskip -1 pt}_f}
\newtheorem{lemma}{Lemma}[section]
\newtheorem{proposition}[lemma]{Proposition}
\newtheorem{theorem}[lemma]{Theorem}
\newtheorem{corollary}[lemma]{Corollary}

\newtheorem{remark}[lemma]{Remark}
\newtheorem{example}[lemma]{Example}

\usepackage{amsmath}
\usepackage{cancel}
\usepackage{mathtools}
\usepackage[margin=1in]{geometry}
\usepackage{courier}
\usepackage[version-1-compatibility]{siunitx}

\usepackage{tikz}
\usetikzlibrary{patterns}
\usetikzlibrary{arrows}
\usetikzlibrary{arrows,backgrounds,snakes}
\tikzset{%
    body/.style={inner sep=0pt,outer sep=0pt,shape=rectangle,draw,thick,pattern=north east lines wide},
    dimen/.style={<->,>=latex,thin,every rectangle node/.style={fill=white,midway,font=\sffamily}},
    symmetry/.style={dashed,thin},
}
\newcommand\getcurrentref[1]{%
 \ifnumequal{\value{#1}}{0}
  {??}
  {\the\value{#1}}%
}

    \numberwithin{common}{section}

\makeatletter
    {\endtcolorbox}
    \numberwithin{common}{section}

\makeatletter

\makeatletter
\renewenvironment{proof}[1][\proofname]{\par
  \pushQED{\qed}%
  \normalfont \topsep6\p@\@plus6\p@\relax
  \trivlist
  \item[\hskip\labelsep
         \underline{\textbf{#1:}}]%\ignorespaces% ADDED
}{
  \popQED\endtrivlist\@endpefalse
}
\makeatother

    \numberwithin{common}{section}

    \numberwithin{common}{section}

\newcommand{\Cle}{\lesssim}
\newcommand{\Cleq}{\Cle}
\newcommand{\leC}{\Cle}
\newcommand{\leqC}{\Cle}

\begin{document}

\title{Convective transport in nanofluids: regularity of solutions and \\error estimates
for finite element approximations}

\author{Eberhard B\"ansch and Pedro Morin}

\date{\today}

\maketitle 

\begin{abstract}
We study the stationary version of a thermodynamically consistent variant of
the Buongiorno model describing convective transport in nanofluids. Under some smallness
assumptions it is proved that there exist regular solutions. Based on this regularity
result, error estimates, both in the natural norm as well as in weaker norms for
finite element approximations can be shown. The proofs are based on the
theory developed by Caloz and Rappaz for general nonlinear, smooth problems.
Computational results confirm the theoretical findings.
\end{abstract}

    
%===================================================================
\section{Introduction}\label{Sec:intro}

Nanofluids, i.e. a dilute mixture of a conventional base fluid
and particles of submicron size, have received much attention for
instance as a cooling liquid. This is due to their superior heat
transfer properties. The enhanced heat transfer cannot solely be understood
by the altered heat conducting coefficients of the mixture, but
rather by effects of a heterogeneous distribution of the particles. 
Among the mathematical models to explain such behavior, the
Buongiorno model \cite{Buon:06} has become rather popular. By now,
many simulations are based on this model, see  
\cite{AbbasEtAl:16,AnbuchEtAl:12,HeEtAl:09,NoghEtAl:15,SayyarSaghafian:17, VanakiEtAl:16} 
for a by far not complete list of applications.
In \cite{Bae:18} the mechanism of the enhanced heat transfer for laminar
flow conditions was revealed: strong temperature gradients at a hot wall lead
to reduction of concentration of particles by thermophoresis
there and this in turn reduces
the concentration dependent viscosity of the dispersion. This then alters
the flow profile leading to a stronger convective heat transfer.

To the best of our knowledge, despite its relevance in applications,
there is hardly any rigorous mathematical analysis of the Buongiorno model. 
In \cite{Bae:18} existence of weak solutions was shown using energy techniques.
It was shown that solutions of a decoupled semi-implicit time-discretization 
converge to a solution of the continuous system, thereby also suggesting 
an effective numerical method. 

In \cite{BaeFM:19} existence of solutions to the stationary system was shown.
Interestingly, the proof is somewhat technically more demanding than for
the time-dependent problem.

The objective of the present work is to first show (under some smallness
assumptions) regularity for the stationary problem and then use these
regularity results to prove quasi-optimal error estimates for finite
element approximations of the system. To this end it is shown that the
system can be cast into the general framework of Caloz and Rappaz 
\cite{Rappaz:97} for nonlinear (smooth) problems.

It turns out that the right space for the scalar quantities concentration
and temperature is $W^1_p$ with $p>d$, $d$ the space dimension, 
whereas for the fluid part we can stay
in a Hilbert space setting. 

The rest of the paper is organized as follows.
In Section~\ref{Sec:model} we present the mathematical model and 
set some notation.
In Section~\ref{Sec:regularity} we present the regularity results for the 
solutions of the system of PDEs.
In Section~\ref{Sec:linear} we present a linearization of the problem 
which will allow us to prove the error estimates for a
finite element discretization in Section~\ref{Sec:fe}.
We close this article with some numerical experiments in
Section~\ref{Sec:comp} where we illustrate the orders of convergence, 
as well as the interesting effect of thermophoresis as a means 
to enhance the heat transfer properties.

%===================================================================
\section{The mathematical model}\label{Sec:model}

We consider the stationary system of a variant of the 
four equations, two-phase Buongiorno model \cite{Buon:06} describing the motion of
a nanofluid including concentration transport by thermophoresis. The model
has been slightly  modified to make it thermodynamically consistent, 
see \cite{Bae:18,BaeFM:19}.  In non-dimensional form it reads as follows: Let $\Omega\subseteq \rz^d$,
$d\in\{2,3\}$ be an open, bounded domain with $C^2$ boundary. We look for a concentration field
$\phi$, a temperature $T$ as well as a velocity $\uu$ and a pressure $p$ fulfilling the following system of equations in $\Omega$
(in the distributional sense)
\newcommand{\jp}{\jj}
\begin{equation}\label{Eq:nondim}
\begin{split}
%-- Phi ----
\uu\cdot\nabla \phi 
       + \frac{1}{\Rey\,\Sc}\nabla \cdot \jp &=0, \\[5pt]
%-- Eta T -----
 \uu\cdot\nabla (\eta T)
+ \frac{1}{\Rey\,\Pra\,\Le}\nabla\cdot (T\jp) 
      - \frac{1}{\Rey\,\Pr}\nabla\cdot (\keff\nabla T)&=f,\\[5pt]
%--- U -------
\uu\cdot\nabla (\rho \uu)
+ \frac{1}{\Rey\,\Scf}\nabla\cdot(\uu \otimes \jp)  
    -\frac{1}{\Rey}\nabla\cdot(\muu D(\uu)) +\nabla p + \beta T \ee_g &=\gb,
    \\
    \Div \uu &= 0,
\end{split}
\end{equation}
with $D(\uu) = \nabla \uu + \nabla \uu^T$ and the particles' flux given by
\begin{equation}
\jp := -\Big(\nabla\phi + \phi(1-\phi)\frac{1}{{\NBT}} \frac{\nabla T}{T_0}\Big),
\end{equation}
with
$\NBT$ the ratio of Brownian diffusivity/thermophoretic diffusivity, 
$T_0$ a non-dimensional ambient temperature,
$\Rey$ the Reynolds number, $\Pr$ the Prandtl number, $\Sc, \Scf$ the Schmidt and
fluid Schmidt number, respectively and $\Le$,  the Lewis number.
Buoyancy effects through a Boussinesq approximation are considered with $\beta>0$ and $\ee_g$ denoting a unit vector in the direction of gravity.
The above system must of course be supplemented by appropriate boundary conditions.

The flux $\jp$ is the non convective slip flux consisting on a Brownian part $-\nabla\phi$
and the so called thermophoretic part $\phi(1-\phi)\frac{1}{{\NBT}} \frac{\nabla T}{T_0}$
that drives particles from hot to cold. Phenomenologically this can be explained by
the fact that collisions of the particles with molecules from the base fluid are stronger on
the hot side of the particle than on the cold part, resulting in a net flux from hot to cold.
For more on thermophoresis we refer for instance to \cite{McNab:73}.

The mathematical challenge with the above system lies in the rather strong
nonlinearity. The right space for $\phi, T$ is therefore $W^1_p(\Omega)$ with
$p >d$. However, an energy estimate is only available in $H^1(\Omega)$, see
\cite{Bae:18,BaeFM:19}. To overcome this problem, in Section \ref{Sec:regularity} we prove 
regularity estimates in $W^2_p(\Omega)$ based on some bootstrap arguments and a smallness assumption.

This paves the way to cast the problem in the general framework 
developed in \cite{Rappaz:97} for nonlinear problems. 

%===================================================================
%\section{Some notations}
\paragraph{Notation.}
As usual, Lebesgue spaces are denoted by $L_p(\Omega)$, $1\leq p \leq \infty$
and Sobolev spaces by $W^m_p(\Omega)$, $m\in\nz_0$. If $p=2$, the
notation $H^m(\Omega)$ is used. 
In what follows, scalar quantities will be denoted by normal characters, whereas
vector and tensor valued functions will be denoted by bold characters. Consequently, for instance
${\mathbf L}^p(\Omega) := L^p(\Omega)^d$. Define $\Vv :=\{ \vv\in \mathbf{H}_0^{1,2}(\Omega)
\;|\; \Div \vv =0\}$, where $H_0^{1,2}(\Omega)$ is the closure of $C^\infty_0(\Omega)$ (the space
of test functions) in $H^{1,2}(\Omega)$, as well as $\tilde{\Vv}:=\Vv\cap \mathbf{H}^{3,2}(\Omega)$
with corresponding norm. 
As usual, the pressure space is chosen to be $L_{2,0}(\Omega) := \{ q\in L_2(\Omega) \;|\;
\intO q(x) dx = 0\}.$
The expression $A \Cle B$ will denote $A \le C B$ with a constant $C$ that might depend on the dimension $d$ of the underlying space and also on the norms involved in the expressions $A$ and $B$.

%We will make frequent use of embeddings
%\todo{}. 
%===================================================================
\section{Regularity}\label{Sec:regularity}

Since in the next sections we are concerned with analytical problems, we set all
non-dimensional constants to one for ease of presentation. 
Let $f \in L_p(\Omega)$, $p>d$, $\gb \in \Lltwo(\Omega)$ and
$\jjp:= -\nabla\phi - h(\phi)\nabla T$, where $h(s) = s(1-s)$. 
Now the problem reads (in the distributional sense): Find $\phi, T,\uu, p$
such that
\begin{equation}\label{eqn}
\begin{split}
-\Delta \phi &= \nabla\cdot (h(\phi) \nabla T )- \uu \cdot \nabla \phi  
       \quad \text{in } \Omega, \\[8pt]
-\nabla\cdot(\keff\nabla T) &= f 
        -\nabla \cdot (T\jjp) -\uu \cdot \nabla (\eta T )    \quad\text{in }\Omega, \\[8pt]
- \nabla\cdot(\muu D(\uu)) 
      +\nabla p&=\gb  -\nabla \cdot (\uu \otimes \jjp) - \uu\cdot\nabla (\rho \uu) - T \ee_g
     \quad\text{in }\Omega,\\
\nabla \cdot \uu &=0 \quad\text{in }\Omega.
\end{split}
\end{equation}
For the boundary conditions we choose
$$
  \phi = \phi_D,\quad 
  T = 0, \quad %\text{and} \quad
   \uu = 0 \qquad\text{on }\partial\Omega.
$$

In the above equations, $\eta = 1 +\phi$, $\rho = 1 + \phi$ and
we assume $\phid\in\Wptwo$ with $0\leq \phid \leq 1$ and $\Delta\phid=0$.

The coefficients $\keff, \muu$ fulfill $k(\cdot), \mu(\cdot)\in C^2([0,1])$ and
$$
0<k_0 \leq k(s),\quad 0<\mu_0 \leq \mu(s) \quad\text{for all } s\in [0,1].
$$
We also denote by $h(\cdot),k(\cdot),\mu(\cdot)$ their extensions to $C^2(\rz)$ satisfying
$$
0<\tilde{k}_0 \leq k(s),\quad 0<\tilde{\mu}_0 \leq \mu(s) \quad\text{for all } s\in\rz
$$
and
$$
|D^\ell h(s)| \leq C, \quad |D^\ell k(s)| \leq C \quad \text{for } \ell=0,1,2
\quad\text{and all } s\in\rz.
$$

We define a vector-valued cut-off function, for $R>0$ as follows
\[
\sigma_R : \real^d \to \real^d, \qquad 
\sigma_R(y) = \begin{cases}
y, \quad &\text{if } |y| \le R, \\
\frac y{|y|}R , \quad &\text{if } |y| > R.
\end{cases}
\]
Note that $\sigma_R$ is Lipschitz and 
\[
\partial_j \big(\sigma_R(y)_i\big) = 
\begin{cases}
\delta_{ij}, \quad &\text{if } |y| \le R, \\
\frac R{|y|}\Big(\delta_{ij}-\frac{y_iy_j}{|y|^2}\Big), \quad &\text{if } |y| > R.
\end{cases}
\]

\newcommand{\rhs}{R\!H\!S}
\newcommand{\rhsT}{\tilde f_R}
\newcommand{\rhsu}{\tilde \gb_R}
\newcommand{\rhsphi}{\rhs(\phi)}

For $R>0$ we now consider the regularized problem
\begin{subequations}\label{eqnR}
\begin{align}\label{eqnR:phi}
\begin{split}
-\Delta \phi_R &= \nabla\cdot \sigma_R(h(\phi_R) \nabla T_R )- 
         \uu_R \cdot \nabla \phi_R  
       \quad \text{in } \Omega, \\
        \phi_R &= \phi_D \quad\text{on }\partial\Omega,
\end{split}
\\
\label{eqnR:T}
\begin{split}%\end{equation} %\\[8pt]
-\nabla\cdot(k(\phi_R))\nabla T_R) &= 
\underbrace{f -\nabla \cdot \Big(T_R\underbrace{\big( -\sigma_R \big( h(\phi_R) \nabla T_R \big)
- \nabla \phi_R\big)}_{=:\jjR}\Big) -\uu_R \cdot \nabla (\eta_R T_R )}_{\rhsT}    \quad\text{in }\Omega, \\
           T_R &= 0 \quad\text{on }\partial\Omega, 
\end{split}
\\ %\end{equation} %\\[8pt]
%\begin{equation}
\label{eqnR:u}
\begin{split}
- \nabla\cdot(\mu(\phi_R) D(\uu_R)) 
      +\nabla p_R&=\underbrace{\gb  -\nabla \cdot (\uu_R \otimes \jjR)
- \uu_R\cdot\nabla (\rho_R \uu_R) -T \ee_g}_{\rhsu} \quad\text{in }\Omega,\\
\nabla \cdot \uu_R &=0 \quad\text{in }\Omega,\\
           \uu_R &= 0 \quad\text{on }\partial\Omega.
\end{split}
\end{align}
\end{subequations}
Existence of solutions $\phi_R$, $T_R \in H^{1,2}(\Omega), 
\uu_R\in \Vv:=\{ \vv\in \mathbf{H}_0^{1,2}(\Omega)
\;|\; \Div \vv =0\}$ can be proved as in \cite{BaeFM:19} 
obtaining also the a priori estimate %\todo{check!}
\begin{equation}\label{uniformbound}
\| \phi_R \|_{1,2} + \|T_R\|_{1,2} + \|\uu_R\|_{\Vv}\le 
C (\|\phid\|_{2,2} + \| f \|_{0,2} +\| \gb\|_{0,2})
\end{equation}
with $C$ independent of $R$ and $0\le \phi_R \le 1$, a.e.

The following lemma will be instrumental in proving our regularity results.

%------------------- Lemma variable coefficients  -------------------------------------------
\begin{lemma}\label{L:variablecoefficients}
%Let $k(s),\mu(s)$ be as stated above. Moreover, 
Let $j\in\nz_0$ and let %\todo{right regularity?}
$k(\cdot), \mu(\cdot)\in C^{j+1}(\rz)$ 
with all derivatives up to order $j+1$ bounded and $\partial\Omega$ of class $C^{j+2}$.
Furthermore, let $r> d$ and $1<M,M\dual < r$, with $M\dual = M/(M-1)$,
the Lebesgue dual exponent to $M$. Let $\tilde{f}\in W^j_M(\Omega)$, 
$\tilde{\gb}\in \mathbf{W}^j_M(\Omega)$
and $\phi\in\W^{1+j}_r(\Omega)$. If $T\in \cringle{H}^1(\Omega),\uu \in\Vv$ fulfill
\begin{align}
%\begin{aligned}
\int_\Omega \keff\nabla T\cdot \nabla\psi  &= \int_\Omega \tilde{f}\psi, & &\text{for all }\psi 
       \in C^\infty_0(\Omega),   \label{Eq:weakT}
       \\
\int_\Omega \frac{\muu}{2} D(\uu):D(\vv)  &= \int_\Omega \tilde{\gb}\cdot\vv, & &\text{for all }\vv\in{\cal D}(\Omega)
    = \{\vv\in (C^\infty_0)^d(\Omega)\,|\, \Div\vv=0\}, \label{Eq:weakU}
%\end{aligned}
\end{align}
then $T,\uu \in W^{2+j}_M(\Omega)$ and
$$ \|T\|_{2+j,M} \leq C_M \|\tilde{f}\|_{W^j_M(\Omega)},\quad
 \|\uu\|_{2+j,M} \leq C_M \|\tilde{\gb}\|_{W^j_M(\Omega)},
$$
where $C_M=C_M(\|\phi\|_{1+j,r})$ is a non-decreasing function of
$\|\phi\|_{1+j,r}$.
\end{lemma}

%\todo{check the proof in \cite{Abels:09}}
\begin{proof}
The regularity for $\uu$ follows from the proof of \cite[Lemma 4]{Abels:09}. There, regularity
in $\mathbf{H}^2(\Omega)$ for $j=0,1$ was shown. However, a closer inspection of the 
proof shows that the assertion is also valid for $M$ as above and
arbitrary $j\in\nz_0$. The regularity for $T$ can be shown in the same way with
even some simplifications.
\end{proof}

The first regularity result for the solution of~\eqref{eqn} holds under a smallness assumption on the data of the problem. 
The precise result is the following.

%------------------- Thm regularity -------------------------------------------
\begin{theorem}\label{Thm:regularity}
Given $p>d$ there exists a constant $F_0 > 0$ such that, 
if $\|\phid\|_{2,p} + \| f \|_{0,p} + \| \gb \|_{0,p}< F_0$, 
then there exists $R > 0$ and a solution 
$(\phi_R,T_R,\uu_R)$ of~\eqref{eqnR} which satisfies $\| \nabla T_R \|_{0,\infty}\le R$,
 and consequently $(\phi,T,\uu)=(\phi_R,T_R,\uu_R)$ is also a solution to~\eqref{eqn}.

Moreover, given $\delta> 0$, there exists $F_\delta > 0$ such that whenever 
$\|\phid\|_{2,p}+ \| f \|_{0,p}+ \| \gb \|_{0,2}\le F_\delta \le F_0$ there exists a solution of~\eqref{eqn} with $\| \nabla T\|_{0,\infty}\le \delta$.
\end{theorem}

\begin{proof}
The proof is based on a couple of bootstrap arguments.

We fix $p > d$ and denote $D := \|\phid\|_{2,p} + \| f \|_{0,p} + \| \gb \|_{0,p} < \infty $.
For each $R > 0$, let $(\phi_R,T_R,\uu_R)$ be a solution 
to~\eqref{eqnR} satisfying~\eqref{uniformbound}, and emphasize that the constants involved in the symbols $\Cle$ below do not depend on $R$, $D$, or the problem data $\phid$, $f$, $\gb$, but only on Sobolev embeddings and the constant $C$ from~\eqref{uniformbound}.

By embedding, $\uu_R \in L_6(\Omega)$.
The right-hand side for the $\phi_R$-equation~\eqref{eqnR:phi} is the sum 
of a function in $L_{3/2}(\Omega)$ and the divergence of a function
in $L_{\infty}(\Omega)$ (bounded by $R$). Then by regularity \cite[Theorem 3.29]{Ambrosio:18},  
$\nabla \phi_R\in L_3(\Omega)$. Now, the right-hand side of~\eqref{eqnR:phi} is the sum of a function in $L_2(\Omega)$ and the divergence of a function
in $L_{\infty}(\Omega)$. This sum is now the divergence of a function in $L_6(\Omega)$, so that by the same regularity result \cite[Theorem 3.29]{Ambrosio:18}, $\nabla\phi_R \in L_6(\Omega)$ and repeating this argument, $\nabla\phi_R \in L_M(\Omega)$ 
for all $1\leq M < \infty$ with the estimate
\begin{equation}\label{boundphiR}
\| \nabla \phi_R \|_{0,M} \Cle R + D, %\le C_M (R+D),
\end{equation}
which thereby implies
\begin{equation}\label{boundjR}
\| \jjR \|_{0,M} \Cle R+D . %C_M (R+D).
\end{equation}

Let us now turn to the equations for $T_R$ and $\uu_R$.
Thanks to~\eqref{eqnR:phi} and the definition of
$\eta,\rho$ the right-hand sides can be written as
\begin{align*}
 \rhsT &= f -(\jjR +\eta_R\uu_R)\cdot \nabla T_R , \\
 \rhsu &= \gb -(\jjR +\rho_R\uu_R)\cdot \nabla \uu_R - T_R \ee_g.
\end{align*}
First, we need an intermediate regularity result for $\uu_R$, namely
$\uu_R\in L_M(\Omega)$ for all $1\leq M < \infty$. To this end, we observe
$$
\|\rhsu\|_{0,3/2} \Cle \| \gb\|_{0,3/2} + (\|\jjR\|_{0,6} +
\|\uu_R\|_{0,6})\|\nabla \uu_R\|_{0,2} + \| T_R \|_{0,3/2}.
$$
Lemma \ref{L:variablecoefficients} (for $j=0$, $M=3/2$, $r>3$) and~\eqref{uniformbound} with~\eqref{boundphiR} yield
$$
\|\uu_R\|_{2,{3/2}} \Cleq \|\rhsu\|_{0,3/2} \Cleq
R+D
$$
and therefore $\uu_R\in\mathbf{L}_M$ for all $1\leq M < \infty$ 
(since $\mathbf{W}^2_{3/2}(\Omega)\hookrightarrow \mathbf{L}_M$) and
\begin{equation}\label{bounduR}
 \|\uu_R\|_{0,M} %\leq C_M \|\rhsu\|_{0,3/2} 
 \leqC  R+D.
\end{equation}

Now fix any $M > 2d/(4-d)$.
Using H\"older's inequality for $\tilde M = \frac{2M}{M+2}<2<p$, 
we get
\[
\| \rhsT \|_{0,{\tilde M}} \le \| f \|_{0,{\tilde M}} 
         + \| \nabla T_R \|_{0,2} (\| \jjR \|_{0,M} + \|\uu_R\|_{0,M})
         ,
\]
whence~\eqref{uniformbound} and the bounds~\eqref{bounduR}~\eqref{boundjR} imply
\[
\| \rhsT \|_{0,{\tilde M}} \Cle \| f \|_{0,{\tilde M}} 
+  D (R+D).
\]

Therefore, by Lemma \ref{L:variablecoefficients} applied to $T_R$ we have 
$
\| T_R \|_{2,{\tilde M}} \leC \| f \|_{0,{\tilde M}}  +  D ( R+D)
$,
which by Sobolev embedding implies
\[
\| \nabla T_R \|_{0,{\tilde M^*}} \leC \| f \|_{0,{\tilde M}}  + D (R+D)
\]
for $\tilde M^* = \frac{\tilde M d}{d- \tilde M} > d$ (see Lemma~\ref{L:propertyofM} below).

Let now $d<t<\tilde t := \min\{ p,\tilde M^*\}\le p$ and $1/q:=1/t - 1/{\tilde t} $ then
\begin{align*}
\| \rhsT \|_{0,t} 
&\le  \| f \|_{0,t} + \| \nabla T_R \|_{0,\tilde t} 
(\|\jjR \|_{0,q} + \|\uu_R\|_{0,q})
%\\ &
\leC  \| f \|_{0,p} +\|f \|_{0,p} R + D R^ 2,
\end{align*}
where we have used~\eqref{bounduR}. 

Finally, using again Lemma \ref{L:variablecoefficients} and the embedding $W_t^2(\Omega) \hookrightarrow W^1_\infty(\Omega)$, we have
$\| \nabla T_R \|_{0,\infty} \leC \| T_R \|_{W_t^2(\Omega)}$ so that
\begin{equation}\label{bound of grad T}
\| \nabla T_R \|_{0,\infty} 
    \le D \big(C_1 + C_2 R + C_3 R^2\big) . %\| f \|_{0,p} + C_2 \|f \|_{0,p} R + C_3 R^ 2.
\end{equation}
%Note that the constants $C_1, C_2, C_3$ do not depend on $R$.
We notice that $D \big(C_1 + C_2 R + C_3 R^2\big)$ can be made smaller than $R$ by choosing $D < R / \big(C_1 + C_2 R + C_3 R^2\big)$, so that under this assumption, $\| \nabla T_R \|_{0,\infty} \le R$ and the first assertion follows with $F_0 := 1/C_2 = \sup_{R >0} R / \big(C_1 + C_2 R + C_3 R^2\big)$.

The second assertion is an immediate consequence of~\eqref{bound of grad T}.
\end{proof}

\begin{comment}
\begin{remark}
	It is worth mentioning that the first assertion of the previous theorem holds without the smallness assumption if $d=2$.
	Indeed, 
\end{remark}
\end{comment}

%---------------------- Lemma M ------------------------------
\begin{lemma}\label{L:propertyofM}
If $M > \frac{2d}{4-d}$ then $\tilde M^* = \frac{d\tilde M}{d - \tilde M}$, the Sobolev conjugate of $\tilde M := \frac{2M}{M+2}$ is larger than $d$.
\end{lemma}

\begin{proof}
Let $M > \frac{2d}{4-d}$, then 
\begin{align*}
\tilde M^* &= \bigg(\frac{2M}{M+2}\bigg)^* 
  = \frac{\frac{2M}{M+2}d}{d-\frac{2M}{M+2}} = \frac{2Md}{d(M+2)-2M} > d 
\end{align*}
because $2M > d(M+2)-2M$ if and only if $M(4-d) > 2d$.
\end{proof}

Combining the previous theorem with regularity results we obtain the following corollary.

%---------------------------- Corollary existence of regular solutions --------
\begin{corollary}\label{C:regular}
If $p>d$, there exist positive constants $F_0$ and $C_p$ such that, 
if $\|\phid\|_{2,p} + \| f \|_{0,p} + \| \gb \|_{0,2} \le F_0$, then problem~\eqref{eqn} has a 
(possibly non-unique) solution $(\phi,T,\uu)$ satisfying
\[  
   \| \phi \|_{2,p}  + \| T \|_{2,p} 
  + \| \uu \|_{2,2}  \le C_p (\|\phid\|_{2,p} + \| f \|_{0,p} + \| \gb \|_{0,2}).
\]
\end{corollary}

\begin{proof}
	We let $F_0$ be as in Theorem~\ref{Thm:regularity} and the problem data satisfy $D:= \|\phid\|_{2,p} + \| f \|_{0,p} + \| \gb \|_{0,2} < F_0$. Then, there exists $R > 0$ and a solution $(\phi_R, T_R, \uu_R)$ of~\eqref{eqnR} which satisfies $\| \nabla T_R \|_{0,\infty} \le R$ and~\eqref{boundphiR}--\eqref{bounduR}. Therefore, $\sigma_R(h(\phi_R)\nabla T_R) = h(\phi_R)\nabla T_R$ and this triple $(\phi_R, T_R, \uu_R)$ is also a solution to~\eqref{eqn} which we now call $(\phi, T, \uu)$.
	
	 Moreover, $\phi$ is the solution of $-\Delta \phi = \rhsphi := \nabla\cdot (h(\phi) \nabla T )- \uu \cdot \nabla \phi $, in $\Omega$, $\phi = \phi_D$ on $\partial\Omega$, with $\rhsphi \in L_p(\Omega)$, so that~\cite[Theorem 3.29]{Ambrosio:18} implies that 
	 \[
	 \| \phi\|_{2,p} \leC \| \rhsphi \|_{0,p} \leC D.
	 \]
	 Also, $T$ is a solution of~\eqref{Eq:weakT} with $\tilde f  =  f 
	 -\nabla \cdot (T\jjp) -\uu \cdot \nabla (\eta T ) \in L_p(\Omega)$ whence
	 \[
	 \| T \|_{2,p} \leC D.
	 \]
It remains to show
$H^2$-regularity for $\uu$, which is an immediate consequence of the fact that $\uu$ is a solution to~\eqref{Eq:weakU} with 
$\tilde\gb = \gb - (\jj + (1+\phi)\uu)\cdot \nabla \uu - T \ee_g$, which satisfies
$$
\|\tilde\gb\|_{0,2} \leC \| \gb\|_{0,2} + (\|\jjR\|_{0,6} +
  \|\uu\|_{0,6})\|\nabla \uu\|_{0,3} + \| T \|_{0,2}.
$$
Note that in the proof of Theorem \ref{Thm:regularity} above 
we have already shown that $\uu_R \in W^2_{3/2}(\Omega) \hookrightarrow W^1_3(\Omega)$.
\end{proof}

%------------------- higher regularity ------------------------------
Once regularity in $W^2_p(\Omega), \mathbf{H}^2(\Omega)$ is established, it is
not difficult to get higher regularity, provided data is more regular. This is stated
in the next corollary.

\begin{corollary}[Higher regularity]\label{Coroll:higher}
For $j\in\nz_0$ let $k(\cdot), \mu(\cdot)\in C^{j+1}(\rz)$ 
with all derivatives up to order $j+1$ bounded, $\partial\Omega\in C^{j+2}$ and $p>d$.
If the solution of \eqref{eqn} is regular, i.e.
$\phi, T\in W^2_p(\Omega), \uu\in\mathbf{H}^2(\Omega)$ and data $\phid\in W^{2+j}_p(\Omega)$,
$f\in W^j_p(\Omega)$, $\gb\in\mathbf{H}^j(\Omega)$ then $\phi, T\in W^{2+j}_p(\Omega)$,
$\uu\in\mathbf{H}^{j+2}(\Omega)$ and
$$
\|\phi\|_{2+j,p} + \|T\|_{2+j,p} + \|\uu\|_{2+j,2}
   \leqC \|\phid\|_{2+j,p} + \|f\|_{j,p} + \|\gb\|_{j,2}.
$$
\end{corollary}

\begin{proof}
The proof is based on induction over $j$. The case $j=0$ is the assumption of this corollary, which holds under the smallness assumption of Corollary~\ref{C:regular}.
So let us assume that the statement is correct for $j$. 
We shall then show that it also holds for $j+1$.
First, note that by the induction assumption and because $p>d$
$$
\partial_\beta\phi, \, \partial_\beta T \in L_\infty(\Omega)
$$
for any multiindex $\beta\in\nz_0^d$ with $|\beta|\leq 1+j$ and 
$$
\partial_\beta\phi, \, \partial_\beta T \in L_p(\Omega)
$$
for $\beta\in\nz_0^d$ with $|\beta|\leq 2+j$. Likewise,
$$
\partial_\beta\uu \in \mathbf{H}^2(\Omega)\hookrightarrow \mathbf{L}_\infty(\Omega)
$$
for $|\beta|\leq j$ and
$$
\partial_\beta\uu \in \mathbf{H}^1(\Omega)\hookrightarrow \mathbf{L}_6(\Omega)
$$
for $|\beta|\leq 1+j$.
Recall that $T$ is a solution to~\eqref{Eq:weakT} with 
$\tilde f = f - (\jj + (1+\phi) \uu)\cdot \nabla T = f + \nabla T \cdot \nabla \phi + h(\phi) |\nabla T|^2 -  (1+\phi)\uu\cdot\nabla T$.
Let us now check the regularity of $\partial_\alpha \tilde f$ for $|\alpha|=1+j$:
\begin{equation}\label{Eq:reg1}
\partial_\alpha \tilde f = \partial_\alpha f + \partial_\alpha(\nabla\phi\cdot\nabla T)
 + \partial_\alpha(h(\phi)\nabla T\cdot\nabla T)
 - \partial_\alpha\big((1+\phi)\uu\cdot\nabla T\big).
\end{equation}
Clearly, the first term on the right-hand side is in $L_p(\Omega)$. By Leibniz' formula,
the second term can be written as
$$
\partial_\alpha (\nabla\phi\cdot\nabla T) = \sum_{\beta\leq\alpha}
\begin{pmatrix}
\alpha \\
\beta
\end{pmatrix}
\nabla\partial_\beta \phi\cdot \nabla\partial_{\alpha-\beta}T.
$$
Our first observation yields
$\partial_\alpha (\nabla\phi\cdot\nabla T) \in L_p(\Omega)$.
The last term can be written by Leibniz' formula as
$$
\partial_\alpha \big((1+\phi)\uu\cdot\nabla T) = 
\mathop{\sum_{|\beta_1|+|\beta_2|+|\beta_3|=1+j}}_{\beta_1,\beta_2,\beta_3 \leq \alpha}
\begin{pmatrix}
\alpha \\
\beta_1,\beta_2,\beta_3
\end{pmatrix}
\partial_{\beta_1}(1+\phi)\partial_{\beta_2}\uu\cdot\nabla \partial_{\beta_3}T. 
$$
In view of our first observation, the worst summand in the sum above (if $6<p$) is attained 
for $|\beta_2|=1+j$ with $\partial_{\beta_2}\uu\in\mathbf{H}^1(\Omega)
\hookrightarrow \mathbf{L}_6(\Omega)$,  showing that altogether
$\partial_\alpha \big((1+\phi)\uu\cdot\nabla T) \in L_{\min\{6,p\}}(\Omega)$. Since $|\alpha|=1+j$ was
arbitrary, this shows $\partial_\alpha\tilde f\in L_{\min\{6,p\}}(\Omega)$. 
As an intermediate result one gets
$$
T\in W^{3+j}_{\min\{6,p\}}(\Omega).
$$
Note that, if $d<p\leq 6$ this is already the desired regularity for $T$.

Now we turn our attention to the $\phi$--equation:
$$
-\Delta \phi %= \Div(h(\phi)\nabla T + (1+\phi)\uu) 
= 
\Div(h(\phi)\nabla T) -\uu\cdot \nabla \phi =: \rhsphi. 
$$
With the same arguments as above, one concludes $\partial_\alpha\rhsphi \in L_{\min\{6,p\}}(\Omega)$ for all $|\alpha|\le j+1$ and thus
$$
\phi \in W^{3+j}_{\min\{6,p\}}(\Omega).
$$
The velocity $\uu$ is a solution to~\eqref{Eq:weakU} with $\tilde \gb = \gb + \big(\nabla \phi + h(\phi)\nabla T - (1+\phi) \uu\big)\cdot \nabla \uu - T \ee_g$.
The derivative of the right hand side $\tilde\gb$ for this momentum equation reads
$$
\partial_\alpha \tilde\gb= \partial_\alpha \gb + \partial_\alpha\big[\big(\nabla \phi
    +h(\phi)\nabla T - (1+\phi)\uu\big)\cdot\nabla \uu\big] - \partial_\alpha T \ee_g.
$$
Expanding again the derivative by Leibniz' formula and observing 
$W^1_{\min\{6,p\}}(\Omega)\hookrightarrow L_\infty(\Omega)$ the worst term
is identified to be 
$$
(1+\phi)\uu\cdot\nabla \partial_\alpha\uu \in \mathbf{L}_2(\Omega)
$$
for all $|\alpha| \le 1+j$ and then
$$
\uu \in \mathbf{H}^{3+j}(\Omega),
$$
which is already the desired regularity for $\uu$.

One more sweep of the above arguments, but now using the intermediate
regularity results,  concludes the proof.
\end{proof}

%--------------  Lemma existence of pressure ------------------------
\begin{lemma}[Existence and regularity of the pressure]\label{L:pressure}
Let the assumptions of Lemma \ref{L:variablecoefficients} hold for some $j,r,M$ and let
$\uu\in\Vv$ be a solution of Eq. \eqref{Eq:weakU}. Then there exists a unique 
pressure $p\in L_{2,0}(\Omega)\cap W^{1+j}_M(\Omega)$ fulfilling
$$
\Div(\muu D(\uu)) + \nabla p  = \tilde{\gb}
$$
and $
\| p \|_{1+j,M} \leq C \| \tilde{\gb}\|_{j,M}.
$
\end{lemma}

\begin{proof}
From standard theory it is clear that there exists a unique pressure $p\in L_{2,0}(\Omega)$ such
that the above equation is fulfilled in the distributional sense. Now, shifting
the term $\Div(\muu D(\uu))$ to the right hand side, differentiating the right hand
side successively up to the desired order and using the regularity
of $\phi$ and $\uu$ it follows that $\nabla p \in W^j_M(\Omega)$ and
the estimate is then also immediate.
\end{proof}

%===================================================================
\section{Linearization}\label{Sec:linear}

Let $p>3\geq d$. Define $X:= \Wpone \times \Wpone \times 
\mathbf{H}_0^{1,2}(\Omega) \times L_{2,0}(\Omega)$,
$X_D :=(\phi_D+\Wpone) \times \Wpone \times \mathbf{H}_0^{1,2}(\Omega) \times L_{2,0}(\Omega)$
and  $Y:= \Wponep \times \Wponep \times \mathbf{H}_0^{1,2}(\Omega) \times L_{2,0}(\Omega)$
with $p\dual = p/(p-1)$ the dual Lebesgue exponent to $p$.

We introduce the nonlinear operator $\operator: X \to Y'$ by
\begin{align*}
\langle \operator(\tilde{\phi},T,\uu,p),&(\psi,\varphi,\vv,q)\rangle
:= \intO\nabla\phi\nabla\psi  + \intO h(\phi) \nabla T \cdot\nabla\psi 
        + \intO\uu \cdot \nabla \phi\,\psi \\
&+\intO\keff\nabla T\cdot\nabla\varphi 
       +\intO(\jjp +\eta\uu)\cdot\nabla T\varphi -\intO f \,\varphi \\
&+ \intO \frac{\muu}{2} D(\uu) : D(\vv)
       +\intO(\jjp +\rho\uu)\cdot\nabla\uu\cdot\vv -\intO p\Div\vv -\intO\gb \cdot\vv - \intO T \ee_g \cdot \vv \\
&+\intO q\Div\uu 
\end{align*}
for all $(\psi,\varphi,\vv,q)\in Y$. Here, $\phi := \phi_D + \tilde{\phi}$.

Since $W^1_{p\dual}(\Omega) \hookrightarrow L_q(\Omega)$ with
$1/q = 1 -1/d - 1/p$ and $p>d$ we have
$2/p +1/q <1$. Also $1/6 + 1/p + 1/q <1$ as well as
$1/p + 1/2 + 1/6 <1$. Thus the above integrals are well defined and also taking
into account the definition of the coefficients, one concludes that 
$\operator: X\rightarrow Y'$ is continuous.
Clearly, $(\phi,T,\uu,p)\in X_D$ is a solution of  system \eqref{eqn}, 
iff $\operator(\tilde{\phi},T,\uu,p)=0$ and $0\leq \phi \leq 1$.

Due to the properties of the coefficients and since $p>3\geq d$ (which in particular
implies $W^1_p(\Omega)\hookrightarrow L_\infty(\Omega)$) $\operator$ is Frechet
differentiable with derivative given by
\begin{equation}\label{Eq:derivative}
\begin{split}
\langle D \operator(\tilde{\phi},&T,\uu,p)(\chi,\Theta,\ww,r),(\psi,\varphi,\vv,q)\rangle
\\
:={}& \intO\nabla\chi\cdot\nabla\psi  + \intO h'(\phi) \chi \nabla T \cdot\nabla\psi 
                           + \intO h(\phi) \nabla \Theta \cdot\nabla\psi \\
       & \hskip 1cm + \intO\uu \cdot \nabla \chi\,\psi + \intO\ww \cdot \nabla \phi\,\psi \\[14pt]
&+\intO\nabla(\keff \Theta)\cdot\nabla\varphi + 
          \intO k'(\phi)(\chi \nabla T- \Theta\nabla \phi)\cdot\nabla \varphi \\
     &  \hskip 1cm +\intO(\partial \jjp +\chi\uu + \eta\ww)\cdot\nabla T\varphi 
       +\intO(\jjp +\eta\uu)\cdot\nabla \Theta\varphi \\[14pt]
&+ \intO \frac{\muu}{2} D(\ww) : D(\vv) + \intO\frac{\mu'(\phi)}{2}\chi D(\uu) : D(\vv) \\
     &\hskip 1cm  +\intO(\partial\jjp +\chi\uu + \rho\ww)\cdot\nabla\uu\cdot\vv 
       +\intO(\jjp +\rho\uu)\cdot\nabla\ww\cdot\vv 
               -\intO r\Div\vv-\intO \Theta \ee_g\cdot \vv \\[14pt]
&+\intO q\Div\ww 
\end{split}
\end{equation}
for all $(\chi,\Theta,\ww,r) \in X, (\psi,\varphi,\vv,q)\in Y$.
Here, $\partial\jjp$ is an abbreviation for
$\partial\jjp= -\nabla\chi - h'(\phi)\chi \nabla T - h(\phi)\nabla\Theta$.

We state the differentiability  and also the Lipschitz continuity 
of $D\operator$ in the 
next lemma.
\bigskip

\begin{lemma}\label{L:differentiability}
Let $X,Y$ and $\operator$ be as above. Then $\operator: X\rightarrow Y'$ is Frechet
differentiable with $D\operator$ given in \eqref{Eq:derivative} above. Moreover,
$D\operator$ is locally Lipschitz continuous, i.e. for 
$U=(\tilde\phi,T,\uu,p)\in X$ and $r_0>0$ there is an $L=L(U,r_0)\geq 0$ such that
$$
\| D\operator (U)- D\operator(V)\|_{{\cal L}(X,Y')} \leq L\|U-V\|_X
$$
for all $V\in B_{r_0}(U)\subseteq X$.
\end{lemma}

\begin{proof}
The differentiability was already discussed above. To show the Lipschitz continuity
we have to estimate
$$
\langle (D\operator (U)- D\operator(V))(\chi,\Theta,\ww,r),(\psi,\varphi,\vv,q)\rangle
  \leq L\|U-V\|_X \|(\chi,\Theta,\ww,r)\|_X \|(\psi,\varphi,\vv,q)\|_Y
$$
which again follows by inspecting the individual integrals and noting that $p>d$.
\end{proof}

The goal of these linearization results is to obtain estimates for a finite element discretization using the framework from~\cite{Rappaz:97}. 
The crucial step now is to show that 
under some smallness assumptions
$D\operator(U)$ is an isomorphism from $X$ to $Y'$.

%----------------------------- Proposition Isomorphism ------------------------
\begin{proposition}\label{P:iso}
There exists $\delta >0$ such that if $\|(\phi,T,\uu,p)\|_X < \delta$, then
$D\operator(\tilde\phi,T,\uu,p): X \rightarrow Y'$ is an isomorphism.
\end{proposition}

\begin{proof}
Given $U=(\tilde{\phi},T,\uu,p)$,
we first consider the reduced operator ${\cal T}: X \rightarrow  Y'$ defined
by
\begin{align*}
\langle {\cal T}(\chi,\Theta,\ww,r),(\psi,\varphi,\vv,q)\rangle
:={}& \intO\nabla\chi\cdot\nabla\psi  + \intO h(\phi) \nabla \Theta \cdot\nabla\psi \\
&+\intO\nabla(R \Theta)\cdot\nabla\varphi  \\
&+ \intO \frac{\muu}{2} D(\ww) : D(\vv) 
               -\intO r\Div\vv\\
&+\intO q\Div\ww 
\end{align*}
with $R: \cringle{W}^1_p(\Omega) \rightarrow \cringle{W}^1_p(\Omega)$, $R\Theta := \keff \Theta$, see
\cite{Rappaz:97}. Due to the properties of $k(\cdot)$ and since $\phi\in W^1_p(\Omega)$,
$R$ is well defined and an isomorphism.

It it easy to show that $\cal T$ is an isomorphism. For this, let
$(l_\phi,l_T,\mathbf{l}_\uu,l_p)\in Y'$ be given. Clearly, by standard theory for
the Stokes equations, there is a unique 
$(\ww, r)\in \mathbf{H}_0^{1,2}(\Omega) \times L_{2,0}(\Omega)$ such that
$$
\langle {\cal T} (0,0,\ww,r),(0,0,\vv,q)\rangle 
=\langle (0,0,\mathbf{l}_\uu,l_p),(0,0,\vv,q)\rangle
$$
for all $(\vv,q)\in \mathbf{H}_0^{1,2}(\Omega) \times L_{2,0}(\Omega)$. 
Next, since the Laplace operator with Dirichlet boundary
condition is an isomorphism from $\cringle{W}^1_p(\Omega)$ 
to $(\cringle{W}^1_{p\dual}(\Omega))'$ \cite[Theorem 1.1]{JK95} and
since $R$ is an isomorphism, there is a unique $\Theta \in \cringle{W}^1_p(\Omega)$
fulfilling
$$
\langle {\cal T} (0,\Theta,0,0),(0,\varphi,0,0)\rangle 
=\langle (0,l_T,0,0),(0,\varphi,0,0)\rangle
$$
for all $\varphi\in \cringle{W}^1_p(\Omega)$. Given these $\Theta,\ww,q$ we finally
can solve the first equation to get a unique $\chi\in\cringle{W}^1_p(\Omega)$ so
that eventually
$$
\langle {\cal T} (\chi,\Theta,\ww,r),(\psi,\varphi,\vv,q)\rangle 
=\langle (l_\phi,l_T,\mathbf{l}_\uu,l_p),(\psi,\varphi,\vv,q)\rangle
$$
for all $(\psi,\varphi,\vv,q)\in Y$. Thus $\cal T$ is bijective. Since
$\cal T$ is continuous, its inverse is also continuous, hence $\cal T$ is
an isomorphism. Because
$X,Y$ are reflexive Banach spaces, it follows that 
there exists $\alpha >0$ such that (see \cite{GuermondErn})
$$
 \mathop{\inf_{U\in X}}_{\|U\|_X=1} \mathop{\sup_{V \in Y}}_{\|V\|_Y=1} 
  \langle {\cal T}U,V\rangle =
 \mathop{\inf_{V\in Y}}_{\|V\|_Y=1} \mathop{\sup_{U \in X}}_{\|U\|_X=1} 
   \langle {\cal T}U,V\rangle = \alpha > 0. 
$$

The remaining part of the operator ${\cal N} := D\operator(\phi,T,\uu,p) - {\cal T}$
reads
\begin{align*}
\langle {\cal N}(\chi,\Theta,\ww,r)&,(\psi,\varphi,\vv,q)\rangle
=  \intO h'(\phi) \chi \nabla T \cdot\nabla\psi 
       + \intO\uu \cdot \nabla \chi\,\psi + \intO\ww \cdot \nabla \phi\,\psi \\ %[14pt]
+{}& \intO k'(\phi)(\chi \nabla T- \Theta\nabla \phi)\cdot\nabla \varphi 
      +\intO(\partial \jjp +\chi\uu + \eta\ww)\cdot\nabla T\varphi 
       +\intO(\jjp +\eta\uu)\cdot\nabla \Theta\varphi \\ %[14pt]
+{}& \intO\frac{\mu'(\phi)}{2}\chi D(\uu) : D(\vv) 
     +\intO(\partial\jjp +\chi\uu + \rho\ww)\cdot\nabla\uu\cdot\vv 
       +\intO(\jjp +\rho\uu)\cdot\nabla\ww\cdot\vv -\intO \Theta \ee_g\cdot \vv.
\end{align*}

The norm of $\cal N$ depends continuously on the $X$-norm of
$(\phi,T,\uu,p)$. Thus, for $(\phi,T,\uu,p)$ sufficiently small we get $ \|{\cal N}\|_{{\cal L}(X,Y')} \le \alpha/2$ and
$$
\mathop{\inf_{U\in X}}_{\|U\|_X=1} \mathop{\sup_{V \in Y}}_{\|V\|_Y=1} 
  \langle D\operator(\tilde\phi,T,\uu,p)U,V\rangle
=
 \mathop{\inf_{V\in Y}}_{\|V\|_Y=1} \mathop{\sup_{U \in X}}_{\|U\|_X=1} 
  \langle D\operator(\tilde\phi,T,\uu,p)U,V\rangle
         \geq \alpha - \|{\cal N}\|_{{\cal L}(X,Y')}
         = \frac\alpha2.
$$
This show that $D\operator(\tilde\phi,T,\uu,p)$ is an
isomorphism.
\end{proof}

%===================================================================
\section{Finite element discretization and error estimates}\label{Sec:fe}

Let $\{\mathcal{T}_h\}_{h>0}$ be a quasiuniform, shape regular family of conforming
triangulations of $\Omega$ with $\max_{T\in \mathcal{T}_h} \text{diam}(T) \le h$.

To avoid technical details estimating the mismatch of the triangulation
with the exact geometry we (unrealistically) assume that elements on
the boundary are curved and match the boundary exactly. 
Hence
$$
\bigcup_{T\in\mathcal{T}_h}T = \bar{\Omega}.
$$

\begin{remark}
The quasiuniformity of $(\mathcal{T}_h)_{h>0}$ is required to guarantee the
$W^1_p$ stability of the Ritz operator (see below).
\end{remark}

To discretize $\cringle{W}^1_p(\Omega),\cringle{W}^1_{p\dual}(\Omega)$ 
we choose Lagrange elements
of polynomial order $k\geq 1$. Denote this space by
$S_h=S_h(\mathcal{T}_h)$. Other choices, however, are possible
and are restricted only by the assumptions in \cite[Chapter 8]{BrennerScott:08}.
Furthermore, for the Navier--Stokes part of the system
we choose an inf-sup stable pair of elements $\Vvh\times Q_h$ with
$\Vvh \subseteq \cringle\mathbf{H}^{1,2}(\Omega)$, $Q_h\subseteq L_{2,0}(\Omega)$
with the approximation property
\begin{equation}\label{Eq:appox}
\inf_{\vv_h\in\Vvh}\|\uu - \vv_h\|_{1,2} 
  + \inf_{q_h\in Q_h} \|p-q_h\|_{0,2} \leqC h^k (\|\uu\|_{k+1,2} + \|p\|_{k,2}).
\end{equation}

The discrete problem reads: Find $U_h=(\tilde{\phi}_h,T_h,\uu_h,p_h)\in
X_h= S_h\times S_h \times \Vvh \times Q_h$ such that
\begin{equation}\label{Eq:discrete}
 \langle \operator(U_h),V_h \rangle =0
\end{equation}
for all $V_h=(\psi_h,\varphi_h,\vv_h,q_h)\in 
Y_h:= S_h\times S_h \times \Vvh \times Q_h$.

%-------------------------------------------------------------------------
\subsection{Error estimates in the norm of $X$.}

Let $U=(\tilde\phi,T,\uu,p)\in X$ be a solution of $\operator(U)=0$.
To be able to apply the general results from \cite{Rappaz:97}
regarding error estimates, we have to show the following.
\begin{enumerate}[(1)]
  \item  $\operator: X\rightarrow Y'$ is differentiable;
  \item $D\operator$ is locally Lipschitz continuous at $U$;
  \item $D\operator(U): X \rightarrow Y'$ is an isomorphism;
  \item dim $X_h$=dim$Y_h$;
  % \item $\lim_{h\rightarrow 0} inf_{U_h\in X_h}\|U-U_h\|_X=0$;
  \item the following discrete inf-sup condition holds:%\\[10pt]
      %\hskip 3cm   
       {$\displaystyle \mathop{\inf_{U_h\in X_h}}_{\|U_h\|_X=1}
       \mathop{\sup_{V_h\in Y_h}}_{\|V_h\|_Y=1}
           \langle D\operator(U)U_h,V_h\rangle
               = \beta >0.$}
\end{enumerate}

Properties (1)--(3) have been shown in the previous section under some smallness
assumption and (4) holds by construction.
The remaining point thus is the inf-sup condition (5).

As in the proof of Proposition \ref{P:iso} we first consider
the reduced operator ${\cal T}$.
Since we chose a pair of finite element spaces $\Vvh\times Q_h$ which is inf-sup stable for Navier-Stokes and due to Korn's inequality, for the $(\uu,p)$ part of ${\cal T}$ one has:
\begin{align*}
 \mathop{\sup_{\|\vv_h\|_{1,2}=1}}_{\|q_h\|_{0,2}=1}
   \langle {\cal T}(0,0,\ww_h,r_h),(0,0,\vv_h,q_h) \rangle 
&=  \mathop{\sup_{\|\vv_h\|_{1,2}=1}}_{\|q_h\|_{0,2}=1}
 \intO \frac{\muu}{2} D(\ww_h) : D(\vv_h) 
       -\intO r_h\Div\vv_h + \intO q_h\Div\ww_h\\[10pt]
  &  \geq \beta_1\Big( \|\ww_h\|_{1,2} + \|r_h\|_{0,2}\Big)
\end{align*}
for all $\ww_h\in \Vvh, r_h\in Q_h$ and some $\beta_1 >0$.

Next we consider the $\Theta$--part of $\cal T$. From \cite[Thm. 10.1]{Rappaz:97}
we infer that
$$
\mathop{\sup_{\varphi_h\in S_h}}_{\|\varphi_h\|_{1,{p\dual}}=1}
   \langle {\cal T}(0,\Theta_h,0,0),(0,\varphi_h,0,0) \rangle 
= \mathop{\sup_{\varphi_h\in S_h}}_{\|\varphi_h\|_{1,{p\dual}}=1}
\intO \nabla(R\Theta_h)\cdot \nabla \varphi_h \geq \beta_2\|\Theta\|_{1,p} 
$$
for all $\Theta_h\in S_h$ and some $\beta_2 >0$.
The crucial point in the proof in \cite{Rappaz:97}
was the stability of the Ritz-operator
$R_h:\cringle{W}^1_p(\Omega) \rightarrow S_h$ in the $W^1_p$-norm:
$$
\|R_h \Theta \|_{1,p} \leqC \|\Theta \|_{1,p}.
$$
In \cite{Rappaz:97} this result was cited from \cite{RannacherScott:82}, where it was proved
for dimension $d=2$. A much more general result valid for $d=2$ as well
as for $d=3$ and a variety of finite element spaces can be found
in \cite{BrennerScott:08}.

In order to get an inf-sup estimate for the $\chi$-equation, we rescale
the $T$-equation for $\operator$: For $\lambda > 0 $ define
$$
\langle \operator_\lambda (\tilde{\phi},T,\uu,p),(\psi,\varphi,\vv,q)\rangle
 :=
\langle \operator (\tilde{\phi},T,\uu,p),(\psi,\lambda\varphi,\vv,q)\rangle.
$$
All what have been shown for $\operator$ and $D\operator$ remains
valid also for $\operator_\lambda$ and $D\operator_\lambda$ except that
$\beta_2$ becomes $\lambda\beta_2$. Now for the $\chi$-equation we use 
the same result from \cite{Rappaz:97} as for the
$\Theta$-equation and infer for $\chi_h, \Theta\in S_h$ 
\begin{align*}
 \mathop{\sup_{\psi_h\in S_h}}_{\|\psi_h\|_{1,{p\dual}}=1}
   \langle {\cal T}_\lambda(\chi,\Theta_h,0,0),(\psi_h,0,0,0) \rangle 
&\geq 
\mathop{\sup_{\psi_h\in S_h}}_{\|\psi_h\|_{1,{p\dual}}=1}
\Big(\intO \nabla\chi_h\cdot \nabla \psi_h 
  - |\intO h(\phi)\nabla \Theta_h\cdot\nabla\psi_h|\Big)\\[8pt]
&\geq \beta_1\|\chi_h\|_{1,p} - c\|\Theta_h\|_{1,p},
\end{align*}
where the last step follows from the boundedness of $h(\cdot)$ and 
H\"older's inequality.

Putting everything together we arrive at
$$
3\times\sup_{\|V_h\|_Y=1}
   \langle {\cal T_\lambda}U_h,V_h \rangle 
 \geq \beta_1 \|\chi_h\|_{1,p}
 + (\lambda\beta_2-c)\|\Theta_h \|_{1,p} 
  + \beta_3 \Big( \|\ww_h\|_{1,2} + \|r_h\|_{0,2}\Big) \geq \beta \|U_h\|_X
$$
for all $U_h=(\chi_h,\Theta_h,\ww_h,r_h)\in X_h$ and some $\beta>0$, 
provided $\lambda$ is sufficiently big.

The rest of the inf-sup estimate follows exactly as in the 
proof of Proposition \ref{P:iso}: define
${\cal N_\lambda} := D\operator_\lambda(\tilde\phi,T,\uu,p) - {\cal T}_\lambda$. Then we
readily have
$$
 \mathop{\inf_{U_h\in X_h}}_{\|U_h\|_X=1} \mathop{\sup_{V_h \in Y_h}}_{\|V_h\|_Y=1} 
  \langle D\operator_\lambda(\tilde\phi,T,\uu,p)U_h,V_h\rangle
         \geq \beta - \|{\cal N}\|_{{\cal L}{(X,Y')}}
         = \tilde{\beta}>0
$$
if $\|(\phi,T,\uu,p)\|_X$ is small enough.

We are now in a state to apply \cite[Thm. 7.1]{Rappaz:97}.

%------------------ Abstract error estimate ------------------------------
\begin{theorem}
Let $p>d$ and $S_h,\Vvh,Q_h$ as above. Let $U=(\tilde\phi,T,\uu,p)\in X$
be a solution of $\operator(U)=0$. 
Then there exist constants $\delta, h_0,r_0,C >0$ such that  if
$\|(\phi,T,\uu,p)\|_X < \delta $ 
for $0< h \leq h_0$ the discrete problem 
\eqref{Eq:discrete} has a locally unique solution $X_h\ni U_h \in B_{r_0}(U)\subseteq X$ and
$$
\|U-U_h\|_X \leq C \inf_{V_h\in X_h}\|U-V_h\|_X.
$$
\end{theorem}

%------------------ final error estimate ------------------------------
\begin{corollary}
Let $p>3\geq d$ and $j\in\nz_0$. Furthermore, let $k=1+j$. The spaces $X,Y,X_h=Y_h$ are as above.
Let $\partial\Omega, k(\cdot), \mu(\cdot)$ fulfill the regularity assumptions
of Corollary \ref{Coroll:higher} and  also
$$
\phid \in W^{2+j}_p(\Omega), \quad f,\gb \in W^j_p(\Omega).
$$
Then there are constants $\delta,h_0,r_0,C>0$ such that the following holds: if
$$
   \|\phid\|_{2,p} + \|f\|_{0,p} + \|\gb\|_{0,2} < \delta
$$
then there is a solution $U\in X$ of problem \eqref{eqn}, i.e. $\operator(U)=0$,  fulfilling
$$
\|\phi\|_{2+j,p} + \|T\|_{2+j,p} + \|\uu\|_{2+j,2}
   \leq C \Big(\|\phid\|_{2+j,p} + \|f\|_{j,p} + \|\gb\|_{j,2}\Big).
$$
Moreover, for $0< h \leq h_0$ the discrete problem 
\eqref{Eq:discrete} has a locally unique solution $X_h\ni U_h \in B_{r_0}(U)\subseteq X$ and
the  following error estimate holds:
$$
\|\phi-\phi_h\|_{1,p} + \|T-T_h\|_{1,p}
  + \|\uu - \uu_h\|_{1,2} + \|p-p_h\|_{0,2} \leqC h^k.
$$
\end{corollary}

\begin{proof}
By the above theorem we know that for sufficiently small $\delta >0$ there is a
solution $(\phi,T,\uu) \in X\cap W^2_p(\Omega) \times W^2_p(\Omega) \times 
\mathbf{H}^2(\Omega)$. Possibly reducing $\delta$ further, the above theorem
guarantees the existence of a locally unique discrete solution $U_h$ and
its quasi-optimality $\|U-U_h\|_X \leqC \inf_{V_h\in X_h}\|U-V_h\|_X$.
The rest follows by the approximation property of the finite element spaces
and the (possibly) higher regularity shown in Corollary \ref{Coroll:higher} and
Lemma \ref{L:pressure}.
\end{proof}

%-------------------------------------------------------------------------
\subsection{Error estimates in weaker norms.}

Error estimates in $L_p, \mathbf{L}_2$ can be proved by using rather standard
duality techniques. To this end,
let us introduce the bilinear form $b(\cdot,\cdot): X\times Y \rightarrow \rz$ for
the solution $U=( \tilde{\phi},T,\uu,p)$:
$$
b(\tilde U,V) := 
\langle D \operator ( \tilde{\phi},T,\uu,p) \tilde U,V \rangle
$$
and the associated operators ${\cal B}: X \rightarrow Y',\; 
{\cal B}^*:Y \rightarrow X'$ by
$$
\langle {\cal B}\tilde U,V \rangle_{Y',Y} = b(\tilde U,V) = \langle \tilde U,{\cal B}^*V\rangle_{X,X'}
$$
for all $\tilde U\in X, V\in Y$.

Recall that ${\cal B}$ is an isomorphism, iff ${\cal B}^*$ is an isomorphism.
As shown in Proposition \ref{P:iso} this is for instance the case, 
if $\|(\phi,T,\uu,p)\|_X$ is small enough, which we assume hereafter.

\bigskip

Now, choose $G=(g_\phi,g_T,\mathbf{g}_\uu,0)\in L_{p\dual}(\Omega)\times
L_{p\dual}(\Omega)\times \mathbf{L}_2(\Omega)\times L_{2,0}(\Omega) \subseteq X'$ such that
$\|g_\phi\|_{0,p\dual} = \|g_T\|_{0,p\dual}=\|\mathbf{g}_\uu\|_{0,2}=1$ and
$$
\langle g_\phi,\phi-\phi_h\rangle = \|\phi-\phi_h\|_{0,p},\quad
\langle g_T,T-T_h\rangle = \|T-T_h\|_{0,p},\quad
\langle \mathbf{g}_\uu,\uu-\uu_h\rangle = \|\uu-\uu_h\|_{0,2}.
$$

Let $W\in Y$ be the unique solution of the dual problem (which exists, since ${\cal B}^*$
is an isomorphism)
\begin{equation}\label{Eq:dual}
b(\tilde U,W) = \langle \tilde U, {\cal B}^*W \rangle_{X,X'} = \langle \tilde U,G\rangle_{X,X'}
\end{equation}
for all $\tilde U\in X$.

Let $W_h\in Y_h$. As in  \cite{Rappaz:97} we calculate
\begin{equation}\label{Eq:aux}
|\langle U-U_h, G\rangle| = |b(U-U_h,W)| \leq | b(U-U_h,W-W_h)| + |b(U-U_h,W_h)|
\end{equation}
for $U,U_h$ the continuous and discrete solution, respectively.
The last term on the right side can be treated as %like \todo{$C^{1,1}$ may also work}
follows
\begin{align*}
b(U-U_h,W_h) &= \langle D \operator ( \tilde{\phi},T,\uu,p)( U-U_h),W_h \rangle \\[8pt]
&= \langle -\operator (U)
+\operator (U_h)
   -D \operator (U) (U_h-U),W_h \rangle \\[8pt]
& = \langle (D\operator ( \zeta(t))- D \operator (U))( U_h-U), W_h\rangle
\end{align*}
for some $t\in[0,1]$ with $\zeta(t) = (1-t)U + t\, U_h$.
Thus  we get
$$
|b(U-U_h,W_h)| \leq L \|U-U_h\|^2_X \|W_h\|_Y,
$$
if $\|U-U_h\|_X$ is sufficiently small and with $L$ the Lipschitz constant of $D\operator$.
The above estimate may be viewed as a substitute for the orthogonality of the error
in the linear case.
Using this estimate in Eq. \eqref{Eq:aux} and the boundedness of $b(\cdot,\cdot)$
(which follows from Lemma \ref{L:differentiability}) we arrive at 
\begin{equation}\label{Eq:aux2}
|\langle U-U_h,G\rangle | \leqC \|U-U_h\|_X\|W-W_h\|_Y + \|U-U_h\|^2_X
  \big( \|W-W_h\|_Y + \|W\|_Y \big).
\end{equation}
It remains to estimate $\|W-W_h\|_Y$, which will be accomplished by
a regularity result for $W$.

%-------------- Regularity estimate for the dual solution
\begin{lemma}\label{L:regdual}
Let the solution $U=(\tilde\phi,T,\uu,p)$ fulfill
$$ 
\tilde\phi, T \in W^2_p(\Omega) \cap \cringle{W}^1_p(\Omega), \quad
 \uu \in \mathbf{H}^2(\Omega) \cap \Vv.
$$
Then the dual solution $W=(\psi,\varphi,\vv,q)$ from Eq. \eqref{Eq:dual} is regular, i.e.
$$
\psi,\varphi \in W^2_{p\dual}(\Omega), \quad \vv \in  \mathbf{H}^2(\Omega), q\in H^1(\Omega)
$$
and 
$$
\| \psi \|_{2,p\dual} +  \| \varphi \|_{2,p\dual} + \| \vv \|_{2,2} + \|q\|_{1,2} \leq C
$$
independent of $h>0$.
\end{lemma}

\begin{proof}
The dual solution $W=(\psi,\varphi,\vv,q)$ fulfills $B^\ast W =G$. This
implies in the distributional sense:
\begin{equation}\label{Eq:regdualphi} 
\begin{split}
%-\Div( (1+h(\phi))&\nabla \psi) 
-\Delta \psi &+ h'(\phi)\nabla T \cdot \nabla \psi
   -\uu\cdot\nabla \psi  \\
&+ k'(\phi)\nabla T\cdot\nabla \varphi + \uu\cdot \nabla T \varphi 
+\Div(\varphi \nabla T) - h'(\phi)|\nabla T|^2 \varphi \\
&-h'(\phi)\vv^T(\nabla\uu \nabla T) 
+(\uu\cdot\nabla\uu)\cdot\vv 
+ \Div(\vv^T \nabla\uu) 
+ \frac{\mu'(\phi)}{2} D(\uu):D(\vv) = g_\phi,\\
\end{split}
\end{equation}
\begin{equation}\label{Eq:regdualT} 
\begin{split}
-\keff \Delta\varphi - k'(\phi)\nabla\phi\cdot\nabla\varphi &+ \Div(h(\phi)\varphi\nabla T)
-\ee_g \cdot \vv \\
&-(\jj +\eta\uu)\cdot\nabla\varphi-\Div(h(\phi)\nabla\psi) + \Div(h(\phi)\vv^T\nabla\uu) = g_T,
\end{split}
\end{equation}
\begin{equation}\label{Eq:regdualu} 
\begin{split}
 -\Div(\mu(\phi)D(\vv)) - \nabla \vv(\jj + \rho\uu) + \rho\nabla \uu^T \vv -\nabla q
  & + \psi\nabla\phi +\eta \varphi\nabla T= \mathbf{g}_\uu, \qquad -\Div\vv = 0,
\end{split}
\end{equation}
%\begin{equation}\label{Eq:regdualub} 
% -\Div\vv = 0
%\end{equation}
or equivalently
\begin{equation}\label{Eq:regdualphi2} 
\begin{split}
%-\Div( (1+h(\phi))\nabla \psi) 
-\Delta \phi = l_\phi :={} &g_\phi -h'(\phi)\nabla T \cdot \nabla \psi
   +\uu\cdot\nabla \psi  \\
&- k'(\phi)\nabla T\cdot\nabla \varphi - \uu\cdot \nabla T \varphi 
-\Div(\varphi \nabla T) + h'(\phi)|\nabla T|^2 \varphi \\
&+h'(\phi)\vv^T(\nabla\uu \nabla T) -(\uu\cdot\nabla\uu)\cdot\vv 
- \Div(\vv^T \nabla\uu) - \frac{\mu'(\phi)}{2} D(\uu):D(\vv),
\end{split}
\end{equation}
\begin{equation}\label{Eq:regdualT2} 
\begin{split}
- \Delta\varphi = l_T:={}& \frac{1}{\keff}\Big(g_T+ k'(\phi)\nabla\phi\cdot\nabla\varphi 
        - \Div(h(\phi)\varphi\nabla T) +(\jj +\eta\uu)\cdot\nabla\varphi - \ee_g \cdot \vv\\
&+\Div(h(\phi)\nabla\psi) - \Div(h(\phi)\vv^T\nabla\uu)\Big),
\end{split}
\end{equation}
\begin{equation}\label{Eq:regdualu2} 
\begin{split}
 -\Div(\mu(\phi)D(\vv)) -\nabla q =\mathbf{l}_\uu:={}&\mathbf{g}_\uu+ \nabla \vv(\jj + \rho\uu) 
      - \rho\nabla \uu^T \vv 
   - \psi\nabla\phi -\eta \varphi\nabla T, \qquad  \Div\vv = 0.
\end{split}
\end{equation}
%
%\begin{equation}\label{Eq:regdualub2} 
% \Div\vv = 0.
%\end{equation}

We already now that there is a unique solution $W\in Y$ fulfilling $\|W\|_Y \leq C$, which
means $\|\psi\|_{1,p\dual}+\|\varphi\|_{1,p\dual}+ \|\vv\|_{1,2}+\|q\|_{0,2} \leq C$.
Let us start inspecting Eqs.~\eqref{Eq:regdualu2}. %, \eqref{Eq:regdualub2}.
Following from the assumption on $(\phi,T,\uu)$ and because 
$W^1_{p\dual}(\Omega) \hookrightarrow L_{q}(\Omega)$ with $q \geq 3/2$ we conclude
$\mathbf{l}_\uu \in \mathbf{L}_{3/2}(\Omega)$. From Lemma \ref{L:variablecoefficients} one infers 
$\vv \in W^2_{3/2}(\Omega)$.

Next, it is readily seen that $l_\phi\in L_{p\dual}(\Omega)$ (note that $p\dual < 2$). 
We show this for the worst term occurring in $l_\phi$, namely $\Div(\varphi\nabla T)$:
$$
\Div(\varphi\nabla T) = \underbrace{\nabla\varphi}_{\in L_{p\dual}(\Omega)}
  \cdot\underbrace{\nabla T}_{\in L_\infty(\Omega)} 
  + \underbrace{\varphi}_{\in L_q(\Omega)}\underbrace{\Delta T}_{\in L_p(\Omega)}
$$
with $1/q = 1/{p\dual} - 1/d$ by embedding and then
$$
\varphi \Delta T \in L_s(\Omega)
$$
with $1/s = 1/q + 1/p = 1/{p\dual} -1/d + 1/p = 1 -1/d \leq 2/3$ 
so that $s\geq 3/2 \geq p\dual$.
From $l_\phi\in L_{p\dual}(\Omega)$ we conclude $\psi \in W^2_{p\dual}(\Omega)$.

With this information one checks that also $l_T\in L_{p\dual}(\Omega)$ and therefore
$\varphi \in W^2_{p\dual}(\Omega)$.

As a last step, we go back to Eqs.~\eqref{Eq:regdualu2}. %, \eqref{Eq:regdualub2}. 
Knowing that
$\psi,\varphi \in W^2_{p\dual}(\Omega)\hookrightarrow L_q(\Omega)$ with $q > 3$ we
finally conclude $\mathbf{l}_\uu \in \mathbf{L}_2(\Omega)$ and so 
$\vv\in \mathbf{H}^2(\Omega)$.

In the above arguments it is of course understood that the corresponding norms
are bounded by the right-hand side.
\end{proof}

Putting everything together, we arrive at the following error estimate.

%------------------------- Error estimate in weak norm --------------------
\begin{theorem}
Let $p>3$ and $U=(\tilde\phi,T,\uu,p)\in X$ be a solution of the continuous system 
Eq. \eqref{eqn} with
$$ 
\tilde\phi, T \in W^2_p(\Omega) \cap \cringle{W}^1_p(\Omega), \quad
 \uu \in \mathbf{H}^2(\Omega) \cap \Vv
$$
and $U_h=(\tilde\phi_h,T_h,\uu_h,p_h)\in X_h$ the corresponding discrete solution from
Eq. \eqref{Eq:discrete}, sufficiently close to $U$.  Then
$$
  \|\phi-\phi_h\|_{0,p} +  \|T-T_h\|_{0,p} + \|\uu-\uu_h\|_{0,2}
  \leqC \|U-U_h\|_X^2 + h \|U-U_h\|_X.
$$
\end{theorem}

\begin{proof}
In view of the definition of the dual solution $W$ from Eq. \eqref{Eq:dual} and
the estimate Eq. \eqref{Eq:aux2}, take the infimum over
all $W_h\in Y_h$ in the latter estimate. The assertion then follows by the regularity
of $W$ from Lemma \ref{L:regdual} and the approximation properties of the finite element spaces.
\end{proof}

%===================================================================
\section{Computational results}\label{Sec:comp}

For the computations presented in this section piecewise quadratic finite
elements are used for $\phi$ and $T$ as well as the 
${\cal P}_2\times {\cal P}_1$-Taylor-Hood element for $\uu$ and $p$.
In order to get an interesting computational example, where one can see the effect of 
thermophoresis, we slightly deviate from the set of boundary conditions imposed 
in the theoretical part. Set $\Omega =]0,2[\times ]0,1[$. For the concentration $\phi$
we impose homogeneous Neumann boundary conditions on the whole of $\partial\Omega$
and a mean concentration $\phi_m=0.1$. The temperature $T$ is set to $T=1$ at the left
side wall and $T=0$ at the right side wall. On the remaining parts of $\partial\Omega$
homogeneous Neumann boundary conditions are imposed. For the velocity $\uu$
homogeneous no-slip conditions are chosen except for the upper boundary, where a
slip condition is enforced, i.e. $\uu_2=0$ together with vanishing tangential
stress $\mu(\phi) (\partial_{x_2}\uu_1 + \partial_{x_1}\uu_2 )=0$. Note that this
condition can be realized as a natural boundary condition for the space
$\{ \vv \in \mathbf{H}^1(\Omega) \;|\; \vv_2 =0$ on the upper part of $\partial\Omega$,
$\vv =0$ else on $\partial\Omega\}$.

For the coefficients $\mu(\cdot), k(\cdot)$ we set
$$
\muu = 1 + 39.11\phi + 533.9\phi^2,\qquad \keff = 1 + 4.5503\phi,
$$
similar to those in \cite{Buon:06}, representing fittings from experimental data
for alumina $\text{Al}_2\text{O}_3$ particles.

Let us first consider the case without thermophoretic effects (i.e. $\phi\equiv\phi_m$).
The flow is driven by buoyancy
forces: the liquid heats up at the left lateral wall inducing an upward flow field that
turns to the right at the top of the container, transporting warm liquid to the right,
cold wall, where it cools down and flows downwards. The cold liquid is flowing back 
at the bottom of the container to the left hot wall, see Fig. \ref{Fig:Velocity}, left picture.

Switching on thermophoretic effects, the flow field is strongly enhanced on the upper
boundary. This can be understood by inspecting the concentration field, see Fig. \ref{Fig:Conc}.
The thermophoretic flux 
$\jj_{therm}=-\phi(1-\phi)\frac{1}{{\NBT}}\frac{\nabla T}{T_0}$ pushes concentration 
away from the left, hot and upper walls
(the flux is in direction from hot to cold), thus
decreasing the viscosity there. The opposite effect takes place at the right, cold wall:
concentration is pushed to the cold wall.

\begin{figure}[th]
\begin{center}
    \includegraphics[width=0.45 \textwidth]{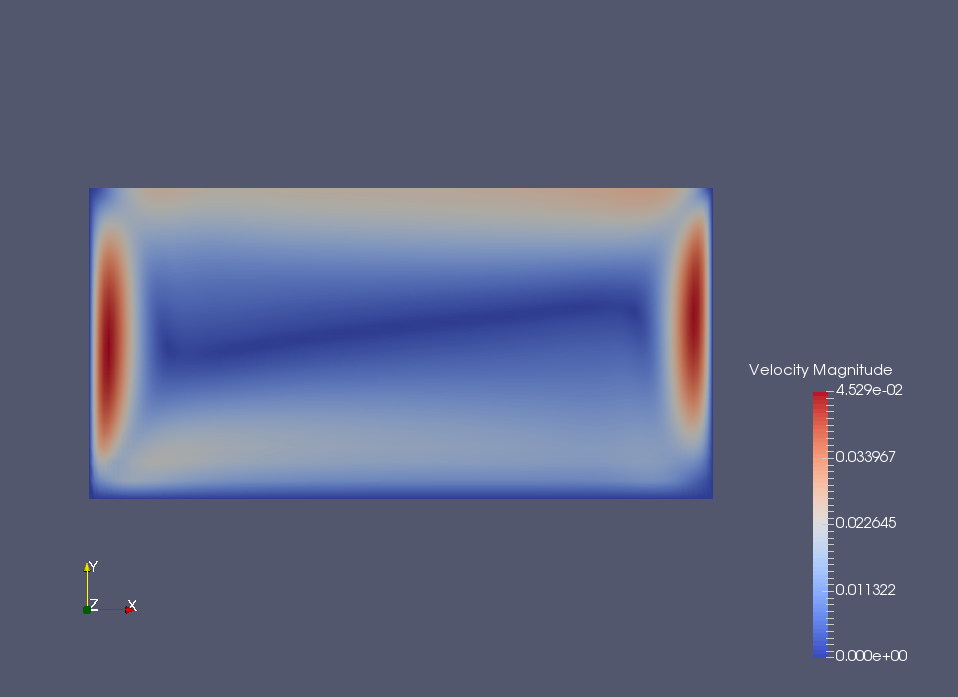}
    \hfil
    \includegraphics[width=0.45 \textwidth]{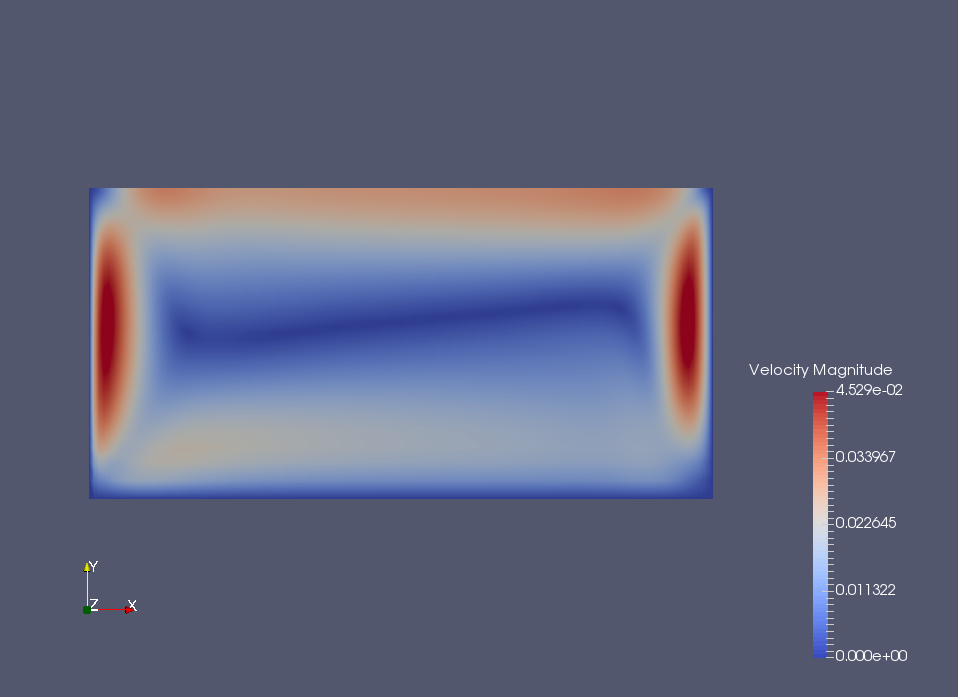}
  \caption{Magnitude of velocity without (left) and with (right)
         thermophoretic effect ($\Rey=700$, $\Pr =6$, $\NBT*T_0=0.586$,
         $\Sc = 1$, $\Scf = 1e10$, $\Le=-1e10$, $\phi_m=0.1$). 
         Switching on thermophoretic effects, the flow field is strongly enhanced on the upper
         boundary.}
\label{Fig:Velocity}
\end{center}
\end{figure}

In order to quantitatively assess the convergence, the same setting as above, however with
different parameters, is used. Since the exact solution is unknown, the computational
solution on a very fine grid with $nt=262,144$ triangles is used as reference instead. Starting from a coarse
triangulation, the grid is successively refined by two bisection steps each. The corresponding
errors and the experimental order of convergence (EOC) are listed in Tab. \ref{Tab:eoc}.
As expected one gets a convergence order of 3 (although the boundary of the domain is not of
class $C^3$).

\begin{figure}[th]
\begin{center}
    \includegraphics[width=0.45 \textwidth]{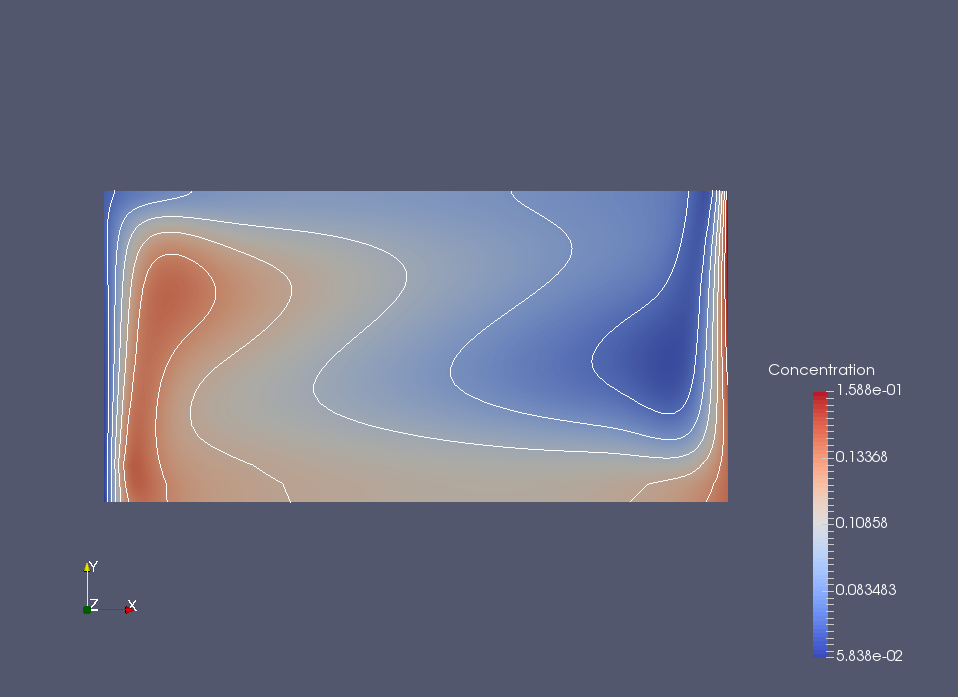}
  \caption{Concentration $\phi$ ($\Rey=700$, $\Pr =6$, $\NBT*T_0=0.586$, 
       $\Sc = 1$, $\Scf = 1e10$, $\Le=-1e10$). 
   The thermophoretic flux 
   $\jj_{therm}=-\phi(1-\phi)\frac{1}{{\NBT}}\frac{\nabla T}{T_0}$ pushes concentration 
   away from the left, hot and upper walls
   (the flux is in direction from hot to cold), thus
   decreasing the viscosity there. The opposite effect takes place at the right, cold wall:
   concentration is pushed to the cold wall.}
\label{Fig:Conc}
\end{center}
\end{figure}

\begin{table}
\begin{center}
{\small
\begin{tabular}{r|c|c|c|c|c|c}
$nt$ & $\|\phi-\phi_h\|_{0,6}$ & EOC & $\|T-T_h\|_{0,6}$ & EOC & $\|\uu-\uu_h\|_{0,2}$ & EOC \\
\hline
 256  &   3.0296e-04 &  ---  &  5.3766e-05 &  ---  &  2.1233e-04 &  ---        \\
1024  &   2.6758e-05 &  3.50 &  5.2109e-06 &  3.37 &  2.7126e-05 &  2.97 \\
4096  &   2.6659e-06 &  3.33 &  5.8438e-07 &  3.16 &  3.3971e-06 &  3.00 \\
16384 &   3.7459e-07 &  2.83 &  7.4084e-08 &  2.98 &  4.2396e-07 &  3.00 \\
\end{tabular}
}
\caption{Errors and EOCs; $nt$ = number of elements.
  $\Rey=100$, $\Pr =1$, $\NBT*T_0=0.586$, $\Sc = 1$, $\Scf = 1e4$, $\Le=1e4$, $\phi_m=0.1$.}
\label{Tab:eoc}
\end{center}
\end{table}
%-------------------------------------------------------------------
%
% References
%
%-------------------------------------------------------------------

\section*{Acknowledgements}
Pedro Morin was partially supported by Agencia Nacional de Promoci\'on Cient\'ifica y Tecnol\'ogica, through grants PICT-2014-2522, PICT-2016-1983, by CONICET through PIP 2015 11220150100661, and by Universidad Nacional del Litoral through grants CAI+D 2016-50420150100022LI. A research stay at Universit\"at Erlangen was partially supported by the Simons Foundation and by the Mathematisches Forschungsinstitut Oberwolfach as well as by the DFG--RTG
2339 \emph{IntComSin}.

\def\cprime{$'$} \def\cprime{$'$} \def\cprime{$'$} \def\cprime{$'$}
  \def\cprime{$'$} \def\cprime{$'$} \def\cprime{$'$} \def\cprime{$'$}

\bibliography{literature}
\bibliographystyle{plain}

\vfill

{\small
\begin{minipage}[t]{0.45\textwidth}
\noindent
Eberhard B\"ansch\\
Applied Mathematics III\\
University Erlangen--N\"urnberg\\
Cauerstr. 11\\
91058 Erlangen\\
Germany\\
{\tt baensch@math.fau.de}
\end{minipage}\hfill
}
\\[1cm]
{\small
	\begin{minipage}[t]{0.45\textwidth}
		\noindent
		Pedro Morin\\
		Facultad de Ingenier\'ia Qu\'imica\\
		Universidad Nacional del Litoral and CONICET \\
		Santiago del Estero 2829 \\
		S3000AOM Santa Fe \\
		Argentina\\
		{\tt pmorin@fiq.unl.edu.ar}
	\end{minipage}\hfill
}

\end{document}